\documentclass[twoside,bezier]{amsart}

\usepackage{epsfig}
\usepackage{amsmath,amsthm,amsfonts,amssymb,amscd,euscript}

\input{epsf}
\newcommand{\filebegin}{\begin{document}}
\newcommand{\fileend}{\end{document}}
\makeatother
\def\thefootnote{}
\newcommand{\lo}{\longrightarrow}
\newcommand{\NMM}{\hspace*{2mm}}

\renewcommand{\baselinestretch}{1.1}

\input{epsf}
\renewcommand{\baselinestretch}{1.1}
\def\n{\noindent}%

\numberwithin{equation}{section}
\def\mapdown#1{\Big\downarrow\rlap
{$\vcenter{\hbox{$\scriptstyle#1$}}$}}
\newcommand{\cmark}{\ding{51}}
\newcommand{\xmark}{\ding{55}}
\newcommand{\lr}{\longrightarrow}
\newcommand{\dis}{\di^{\circledast}}
\newcommand{\fo}{\mathcal{F}}
\newcommand{\frs}[2]{\hbox{$\ds \frac{#1}{#2}$}}
\newcommand{\op}{\overline{\parzial}}
\newcommand{\pa}{\partial}
\newcommand{\paa}[2]{\partial_{#1#2}}
\newcommand{\paaa}[2]{\partial_{#1#2^2}}
\newcommand{\paaaa}[3]{\partial_{#1#2#3}}
\newcommand{\ot}{\otimes}
\newcommand{\om}{\omega}
\newcommand{\hs}[1]{\hspace{#1cm}}
\newcommand{\ric}[4]{#1\omega_{#2}\circ\omega_{#3}\otimes\omega_{#4}}
\newcommand{\rics}[4]{#1\omega_{#2}\ot\omega_{#3}\otimes\omega_{#4}}
\theoremstyle{plain}
\newtheorem{theorem}{Theorem}[section]
\newtheorem*{theorem*}{Theorem}
\newtheorem{definition}[theorem]{Definition}
\newtheorem{lemma}[theorem]{Lemma}
\newtheorem{remark}{Remark}
\newtheorem{prop}[theorem]{Proposition}
\newtheorem{cor}[theorem]{Corollary}
\newtheorem{rem}[theorem]{Remark}
\newtheorem{ex}[theorem]{Example}
\newtheorem*{mt*}{Main Theorem}
\DeclareMathOperator{\di}{d}
\newcommand\C{{\mathbb C}}
\newcommand\f{\mathcal{F}}
\newcommand\g{{\mathfrak{g}}}
\newcommand\h{{\mathfrak{h}}}
\newcommand\ka{{\mathfrak{k}}}
\renewcommand\r{{\mathfrak{r}}}
\newcommand\su{{\mathfrak{su}}}
\renewcommand\k{{\kappa}}
\renewcommand\l{{\lambda}}
\newcommand\m{{\mathfrak{m}}}
\renewcommand\O{{\omega}}
\renewcommand\t{{\theta}}
\newcommand\ebar{{\bar{\varepsilon}}}
\newcommand\R{{\mathbb R}}
\newcommand\Z{{\mathbb Z}}
\newcommand\s{{\mathbb S}}
\newcommand{\de}[2]{\frac{\partial #1}{\partial #2}}
\newcommand\w{\wedge}
\newcommand{\ov}[1]{\overline{ #1}}
\newcommand{\Tk}{\mathcal{T}_{\omega}}
\newcommand{\ovp}{\overline{\partial}}
\newcommand\ad{{\rm ad}}
\newcommand{\D}{\mathcal D}
\newcommand{\A}{\mathcal A}
\newcommand{\B}{\mathcal B}
\newcommand{\G}[3]{\Gamma_{#1#2}^{#3}}
\renewcommand{\theequation}{\thesection.\arabic{equation}}
\begin{document}
\vspace*{2cm}
\begin{center}
{\bf\large On Lorentzian two-symmetric manifolds of dimension-four}
 \\[0.5cm]
{Amirhesam Zaeim, Mohammad Chaichi and Yadollah Aryanejad \\
Department of Mathematics, Payame noor University,\\ P.O. Box 19395-3697, Tehran, Iran\\ zaeim@pnu.ac.ir\\ chaichi@pnu.ac.ir \\ y.aryanejad@pnu.ac.ir} \\[2mm]

\end{center}
\vspace*{0.5cm}
\begin{quotation}
\noindent
{\footnotesize
{\sc Abstract.}
We study curvature properties of four-dimensional Lorentzian manifolds with two-symmetry property. We then consider Einstein-like metrics, Ricci solitons and homogeneity over these spaces.}
\end{quotation}

\n \textbf{2000 Mathematics subject classification: } 53C50, 53C15, 53C25.

\markboth
{A.Zaeim, M.Chaichi and Y.Aryanejad}
 {On Lorentzian two-symmetric manifolds of dimension-four}

\section{Introduction}
Symmetries of the mathematical models have a lot of applications in applied sciences. For example, molecular symmetries  studied in \cite{iran1} and \cite{iran2}, obtaining the group of symmetries of the molecules.
 $K$-symmetry spaces are a natural generalization
of
symmetric manifolds. A (pseudo-) Riemannian space $(M, g)$ is called $k$-symmetric if the following condition is valid:
$$
\nabla^kR=0,\quad\nabla^{k-1}R\neq0,
$$
where $k\geq1$ and $R$ is the curvature tensor of $(M,g)$. In the Riemannian setting, contrary to the pseudo-Riemannian case, a $k$-symmetric space is necessarily locally symmetric, i.e., $\nabla R=0$ \cite{Ta}. Examples of pseudo-Riemannian $k$-symmetric spaces with $k\geq2$ can be found in \cite{Se,Bl,Ka}. Many interesting results about Lorentzian two-symmetric spaces were presented in \cite{Se}, in particular the author proved that any two-symmetric Lorentzian manifold admits a parallel null vector field. A classification of four-dimensional two-symmetric Lorentzian spaces is obtained in \cite{Bl}, based on the Petrov classification of the Weyl tensors, and it is shown that such spaces are some special pp-waves. For wide applications in physics, many authors studied pp-wave manifolds which are spacial kind of pr-wave spaces. In \cite{Ba}, local symmetry, conformal flatness, Einstein-like metrics and existence of non-trivial Ricci solitons studied on the conformally flat pr-wave manifolds. Homogeneous plane wave manifolds, other special kind of pr-waves, investigated in \cite{Blau} and one geodesically complete family of the spaces under consideration were found. The generalization of the results of \cite{Bl} is the subject of \cite{Al}, where it is proven that a locally indecomposable Lorentzian manifold of dimension $n+2$ is two-symmetric if and only if there exist local coordinates $(v,x^1,\dots,x^n,u)$ such that
\begin{equation}\label{g0}
g=2dvdu+\sum_{i=1}^n(dx^i)^2+(H_{ij}u+F_{ij})x^ix^j(du)^2,
\end{equation}
where $H_{ij}$ is a nonzero diagonal matrix with diagonal elements $\lambda_1\leq\dots\leq\lambda_n$, and $F_{ij}$ is a symmetric real matrix. According to this general form of Lorentzian two-symmetric manifolds, in the four-dimensional case, there exist local coordinates $(x^1,\dots,x^4)$ such that the metric $g$ of a Lorentzian two-symmetric space is
\begin{equation}\label{g}
g=2dx^1dx^4+(dx^2)^2+(dx^3)^2+\left(x^4(a(x^2)^2+b(x^3)^2)+p(x^2)^2+2qx^2x^3+s(x^3)^2\right)(dx^4)^2,
\end{equation}
where $a,b,p,q,s$ are real constants and $a^2+b^2\neq0$. Our main goal is to study some geometric properties of four-dimensional Lorentzian two-symmetric spaces.

This paper is organized in the following way. Curvature properties of Lorentzian two-symmetric four-spaces will be studied in the section two and Einstein-like metrics of the spaces under consideration is the subject of section three. Ricci solitons and homogeneous four-dimensional Lorentzian two-symmetric spaces will be studied in section four and five respectively.

\section{Two-symmetric Lorentzian four-manifolds}
The first step for study the geometry of  (pseudo-)Riemannian manifolds is to determine the Lievi-Civita connection. By using the {\em Koszul} identity $2g(\nabla_XY,Z)=Xg(Y,Z)+Yg(Z,X)-Zg(X,Y)-g(X,[Y,Z])+g(Y,[Z,X])+g(Z,[X,Y]),$ and applying the metric \eqref{g}, one can determine the components of the Levi-Civita connection. We use $\pa_i=\frac{\pa}{\pa x^i}$ as a local basis for the tangent space and have:
\vspace{.3cm}
\begin{theorem}
Let $(M,g)$ be an arbitrary two-symmetric Lorentzian four-manifold, where the metric $g$ is described in local coordinates $(x^1,x^2,x^3,x^4)$ by the Equation \eqref{g}. The non-zero components of the Levi-Civita connection are:
\begin{equation}
\begin{array}{l}
\nabla_{\pa_2}\pa_4=(ax^2x^4+px^2+qx^3)\pa_1,\\
\nabla_{\pa_3}\pa_4=(bx^3x^4+sx^3+qx^2)\pa_1,\\
\nabla_{\pa_4}\pa_4=\frac{a(x^2)^2+b(x^3)^2}{2}\pa_1-(ax^2x^4+px^2+qx^3)\pa_2-(bx^3x^4+qx^2+sx^3)\pa_3.
\end{array}
\end{equation}
\end{theorem}
Applying the relation $R(X,Y)=[\nabla_X,\nabla_Y]-\nabla_{[X,Y]}$ we immediately determine the curvature tensor. If we  set $R(\pa_k,\pa_l)\pa_j=R^i_{jkl}\pa_i$, then by contraction on the first and  third indices of the curvature tensor, the Ricci tensor $\varrho$ will be deduced. The scalar curvature tensor $\tau$ is also obtained by full contraction of coefficients of the curvature tensor.
\vspace{.3cm}
\begin{theorem}
A four-dimensional two-symmetric Lorentzian space admits zero scalar curvature. Also, the non-zero components of curvature tensor and Ricci tensor are
\begin{equation}\label{R}
\begin{array}{l}
R(\pa_2,\pa_4)=(ax^4+p)\pa_1dx^2+q\pa_1dx^3-(ax^4+p)\pa_2dx^4-q\pa_3dx^4,\\
R(\pa_3,\pa_4)=q\pa_1dx^2+(bx^4+s)\pa_1dx^3-q\pa_2dx^4-(bx^4+s)\pa_3dx^4,\\
\\
\varrho(\pa_4,\pa_4)=-(a+b)x^4-(s+p).
\end{array}
\end{equation}
\end{theorem}

A (pseudo-)Riemannian manifold $(M,g)$ is called {\em Einstein} if $\varrho=cg$, for a real constant $c$. Being {\em Ricci flat} means that the Ricci tensor vanishes identically. Also, {\em conformal flatness} translates into the following system of algebraic equations:
\begin{eqnarray}\label{W}
&W_{ijkh} = R_{ijkh}-\frac{1}{2}(g_{ik}\varrho_{jh}+g_{jh}\varrho_{ik}-g_{ih}\varrho_{jk}-g_{jk}\varrho_{ih})
+\frac{r}{6}(g_{ik}g_{jh}-g_{ih}g_{jk})=0 \label{cflat}\\[6 pt]
&\text{for all indices} \; i,j,k,h=1,\dots,4, \hspace{210pt} \nonumber
\end{eqnarray}
where $W$ denotes the {\em Weyl tensor} and $\tau$ is the scalar curvature. Although two-symmetric spaces clearly aren't flat,  we can check Ricci flatness.
\vspace{.3cm}
\begin{theorem}
Let $(M,g)$ be a two-symmetric four-dimensional Lorentzian manifold such that the metric $g$ is described by the Equation \eqref{g} in local coordinates $(x^1,x^2,x^3,x^4)$. The following statements satisfy:
\begin{itemize}
\item[a)] $(M,g)$ is Einstein if and only if be Ricci flat if and only if $b=-a,\ s=-p$.
\item[b)] $(M,g)$ is conformally flat if and only if $b=a,\ s=p,\ q=0$.
\end{itemize}
\end{theorem}
\begin{proof}
Let $(M,g)$ be an Einstein manifold. Using  the Equation \eqref{R} we set $\varrho=cg$. The following relations must be  established.
$$
c=(a+b)x^4+s+p=0.
$$
So the Einstein property is equivalent to satisfying $b=-a$, $s=-p$ and $c=0$ which is clearly equivalent to Ricci flatness. Using the Equation \eqref{W}, the non-zero components of the Weyl tensor would be
$$
\begin{array}{ll}
W_{2424}=-W_{3434}=\frac{(b-a)x^4+s-p}{2},& W_{2434}=-q,\\
\end{array}
$$
it is obvious that the Weyl tensor vanishes if and only if $b=a$, $s=p$ and $q=0$.
\end{proof}

\section{Einstein-like Lorentzian two-symmetric spaces}

Two new classes of Riemannian manifolds which are defined through conditions on the Ricci tensor, introduced by A. Gray in \cite{Gr}. These types of manifolds which are famous as $\A$ and $\B$ classes, would be extended at once to the pseudo-Riemannian geometry. $\A$ and $\B$ classes are defined in the following way:

{\bf Class $\A$:} a pseudo-Riemannian manifold $(M, g)$ belongs to class $\A$ if and only if its Ricci tensor $\varrho$ is {\em cyclic-parallel}, that is,
\begin{equation}\label{cyclic}
(\nabla_X \varrho)(Y,Z)+(\nabla_Y\varrho)(Z,X)+(\nabla_Z\varrho)(X,Y)=0,
\end{equation}
for all vector fields $X$, $Y$ and $Z$  on $M$. The Equation \eqref{cyclic} is equivalent to requiring that $\varrho$ is a {\em Killing tensor}, that is,
\begin{equation}\label{Killing}
(\nabla_X\varrho)(X,X)=0.
\end{equation}

To note that Equation \eqref{Killing}, also known as the first odd Ledger condition, is a necessary condition for a (pseudo-)Riemannian manifold to be a D'Atri space. Hence, identifying two-symmetric manifolds of a given dimension satisfying \eqref{Killing}, is the first step to understand D'Atri spaces.

{\bf Class $\B$:} a pseudo-Riemannin manifold $(M,g)$ belongs to class $\B$ if and only if its Ricci tensor be a {\em Codazzi tensor}, that is,
\begin{equation}\label{Codazzi}
(\nabla_X\varrho)(Y,Z)=(\nabla_Y\varrho)(X,Z).
\end{equation}

A pseudo-Riemnnain manifold which belongs to one of the above classes is called an {\em Einstein-like manifold}. We denote the class of Ricci parallel, Einstein and manifolds with constant scalar curvature by $\mathcal P$, $\mathcal E$ and $\mathcal C$ respectively. One can easily see that the intersection of two Einstein-like classes consists of Ricci parallel manifolds. This situation can be summarized in the following diagram:
$$
\begin{array}{ccc}
&\subset\A&\\
\mathcal E\subset\mathcal P=\A\cap\B&&\subset\A\cup\B\subset\mathcal C.\\
&\subset\B&
\end{array}
$$
Einstein-like metrics are deeply investigated through the different kinds of homogeneous spaces in both Riemannian and pseudo-Riemannian signatures. Three-dimensional Riemannian homogeneous spaces studied in \cite{Ab}. In \cite{Bu}, authors study three- and four-dimensional Einstein-like Riemannian manifolds which are Ricci-curvature homogeneous. They could completely classify three-dimensional case of the mentioned spaces, while in the four-dimensional case, they partially classified the special case where the manifold is locally homogeneous. They also presented explicit examples of four-dimensional locally homogeneous Riemannian manifolds whose Ricci tensor is cyclic-parallel and has distinct eigenvalues. These examples invalidated the expectation stated by F. Podest{\'a} and A. Spiro in \cite{Po}. Three-dimensional ball-homogeneous spaces, semi-symmetric spaces, Sasakian spaces and three-dimensional contact metric manifolds are other Riemannian classes which were the subject of research for the Einstein-like properties \cite{Ca1,Bo,Ab1,Ca2}.
\vspace{.3cm}
\begin{theorem}\label{A}
Every four-dimensional two-symmetric Lorentzian manifold $(M,g)$ belongs to class $\A$, if and only if $b=-a$.
\end{theorem}
\begin{proof}
Let $v=v^1\pa_1+v^2\pa_2+v^3\pa_3+v^4\pa_4$ be an arbitrary vector space on $(M,g)$, where $v^1,\dots,v^4$ are smooth functions on $M$. As mentioned before, $(M,g)$ belongs to class $\A$ of Einstein-like manifolds if and only if the Equation \eqref{Killing} satisfies. By straight forward calculations it is implied that
$$
(\nabla_v\varrho)(v,v)=-(v^4)^3(a+b).
$$
Thus, $(M,g)$ belongs to class $\A$ if and only if $b=-a$.
\end{proof}
\vspace{.3cm}
\begin{theorem}
Every four-dimensional two-symmetric Lorentzian manifold $(M,g)$ belongs to class $\B$ of the Einstein-like manifolds.
\end{theorem}
\begin{proof}
Let $v=\sum_{i=1}^4 v^i\pa_i$, $u=\sum_{i=1}^4 u^i\pa_i$ and $w=\sum_{i=1}^4w^i\pa_i$ be three arbitrary smooth vector fields on $(M,g)$. Every  two-symmetric space $(M,g)$ belongs to class $\B$ of Einstein-like manifolds if and only if the Equation \eqref{Codazzi} satisfies. Direct calculations yield that
$$
(\nabla_u\varrho)(v,w)=(\nabla_v\varrho)(u,w)=-u^4v^4w^4(a+b),
$$
and so, the Equation \eqref{Codazzi} always establishes and proves the claim.
\end{proof}
\section{Two-symmetric Lorentzian Ricci solitons}
We now report some basic information on Ricci solitons, referring to \cite{Cao} for a survey and further references. A \emph{Ricci soliton} is a pseudo-Riemannian manifold $(M,g)$  admitting a smooth vector field $V$, such that
\begin{equation}\label{soliton}
\mathcal{L}_V g+\varrho=\lambda g,
\end{equation}
where $\mathcal{L}$ denotes the Lie derivative and $\lambda$ a real constant.
A Ricci soliton is said to be \emph{shrinking}, \emph{steady} or
\emph{expanding} depending on whether $\lambda>0$, $\lambda=0$ or $\lambda<0$, respectively. Ricci solitons are  self-similar solutions of the \emph{Ricci flow}.

Originally introduced in the Riemannian case, Ricci solitons have been intensively studied in pseudo-Riemannian settings in recent years. The Ricci soliton equation is also a special case of Einstein field equations. 
\vspace{.3cm}
\begin{theorem}
Every four-dimensional two-symmetric Lorentzian manifold $(M,g)$ is  shrinking, expanding and steady Ricci soliton.
\end{theorem}
\begin{proof}
Let $(M,g)$ be a four-dimensional two-symmetric Lorentzian manifold, where $g$ is described by the Equation \eqref{g}. Suppose that $v=\sum_{i=1}^4v^i\pa_i$ is a smooth vector field on $(M,g)$ such that the Equation \eqref{soliton} satisfies for a real constant $\lambda$. The Lie derivative of $g$ in direction $v$ is

\begin{tabular}{lp{12cm}}
$\mathcal L_vg=$&$2\pa_1v^4(dx^1)^2+2(\pa_2v^4+\pa_1v^2)dx^1dx^2+2(\pa_3v^4+\pa_1v^3)dx^1dx^3$\\
&$+2(\pa_1x^1+\pa_4v^4+\pa_1v^4(a(x^2)^2x^4+p(x^2)^2+2qx^2x^3+b(x^3)^2x^4+s(x^3)^2))dx^1dx^4
+2\pa_2v^2(dx^2)^2+2(\pa_2v^3+\pa_3v^2)dx^2dx^3
+2(\pa_2v^4(a(x^2)^2x^4+p(x^2)^2+2qx^2x^3+b(x^3)^2x^4+s(x^3)^2+1)+\pa_2v^1)dx^2dx^4
+2\pa_3v^3(dx^3)^2+2(\pa_3v^4(a(x^2)^2x^4+p(x^2)^2+2qx^2x^3+b(x^3)^2x^4+s(x^3)^2+1)+\pa_3v^1)dx^3dx^4
+(\pa_4v^4(2a(x^2)^2x^4+2p(x^2)^2+4qx^2x^3+2b(x^3)^2(x^4)+2s(x^3)^2)+2\pa_4v^1+2av^2x^2x^4+2pv^2x^2
+2qv^2x^3+2qv^3x^2+2bv^3x^3x^4+2sv^3x^3+av^4(x^2)^2+v^4b(x^3)^2)(dx^4)^2.$
\end{tabular}

Applying Equations \eqref{g} and \eqref{R} in the Ricci soliton Equation \eqref{soliton}, we have a system of PDEs which admits the following solution
$$
\left\{\begin{array}{l}
\lambda=2c,\\
v^1=\frac{(x^4)^2}{4}(a+b)+\frac{x^4}{2}(s+p)+2cx^1,\\
v^2=cx^2,\\
v^3=cx^3,\\
v^4=0,
\end{array}\right.
$$
for a real constant $c$. Since $c$ is arbitrary, $(M,g)$ can be an expanding, shrinking or steady Ricci soliton.
\end{proof}
A Ricci soliton $(M,g)$ is called {\em gradient} if  the  Equation~\eqref{soliton} holds for a vector field $X = {\rm grad} f$, for some potential function $f$. In this case, \eqref{soliton} can be rewritten as
$$
2{\rm Hes}_f  +\varrho =  \lambda g,
$$
where ${\rm Hes}_f$ denotes the Hessian of $f$. Studying locally conformally flat Lorentzian gradient Ricci solitons, as well as quasi-Einstein spaces, in \cite{Br} and \cite{Br1} proved that such spaces are locally isometric to a plane-wave, if the gradient of the potential function is null.
\vspace{.3cm}
\begin{theorem}\label{thm3}
Every four-dimensional two-symmetric Lorentzian space $(M,g)$ is a gradient Ricci soliton if and only if be a steady Ricci soliton. In this case, the potential function is $f=\frac{a+b}{12}(x^4)^3+\frac{p+s}{4}(x^4)^2+c_1x^4+c_2$, for arbitrary real constants $c_1,c_2$.
\end{theorem}
\begin{proof}
Let $f=f(x^1,x^2,x^3,x^4)$ be a smooth function on $(M,g)$ and $v=\sum_{i=1}^4v^i\pa_i$ be a gradient Ricci soliton with the potential function $f$. So the coefficient $v^i$ must be $v^i = \sum_{j=1}^4g^{ij}\partial_j(f)$. By applying $v$ to the Equation \eqref{soliton} the following equations must establish
$$\left\{
\begin{array}{l}
f_{11}=f_{12}=f_{13}=f_{23}=0,\\
\lambda=2f_{14}=2f_{22}=2f_{33},\\
2f_{24}-2af_1 x^2x^4-2pf_1x^2-2qf_1x^3=0,\\
2f_{34}-2bf_1x^3x^4-2qf_1x^2-2sf_1x^3=0,\\
\lambda(x^2)^{2}(ax^4+p)+2\lambda qx^2x^3+\lambda(x^3)^{2}(bx^4+s)+(a+b)x^4+s+p-2f_{44}\\
+f_1(a(x^2)^{2}+b(x^3)^{2})-2f_2(ax^2x^4+px^2+qx^3)-2f_3(qx^2+bx^3x^4+sx^3)=0,
\end{array}\right.
$$
where $f_{i}:=\pa_if$. After solving the above system of PDEs we get that $\lambda$ must  be vanished  and $f$ must be the same function of the statement, this matter ends  the proof.
\end{proof}

To note that, from the above Theorem~\ref{thm3} we get $\nabla f=\frac{a+b}{4}(x^4)^2+\frac{p+s}{2}x^4+c_1$, which is a null vector field. This result is compatible with main Theorem in \cite{Br}.
\section{Homogeneous two-symmetric four-dimentional spaces}
Study of homogeneous spaces is one of the most interesting topics in differential geometry, where a deep connection between geometry and algebra appears. A (pseudo-)Riemannian manifold $M$ is called {\em homogeneous}, if for any points $p,q\in M$, there is an isometry $\phi$ of $M$ such that $\phi(p)=q$. In short, $I(M)$, the group of isometries of $M$, acts transitively on $M$. If $M$ is homogeneous, then evidently any geometrical properties at one point of $M$ holds at every point. For a detailed introduction to homogeneous spaces see e.g., \cite{SH,SK,AW}.

Homogeneous Riemannian structures introduced by Ambrose and Singer in \cite{WA} and deeply studied in \cite{FT}. The notation is  generalized to homogeneous pseudo-Riemannian structures by Gadea and Oubi\~na in \cite{PG}, in order to obtain a characterization of reductive homogeneous pseudo-Riemannian manifolds. A pretty application of homogeneous structures on three dimensional Lorentzian manifold is shown  in  \cite{GC}.

Let $(M, g)$ be a connected pseudo-Riemannian manifold, the following definition introduced by Gadea and Oubi\~na:
\begin{definition}\cite{PG}
A homogeneous pseudo-Riemannian structure on $(M, g)$ is a tensor field $T$ of type $(1, 2)$ on $M$,
such that the connection $\widetilde{\nabla}=\nabla -T$ satisfies
\begin{equation}\label{HS}
\begin{array}{ccc}
\widetilde{\nabla}g=0,\quad  \widetilde{\nabla}R=0,\quad  \widetilde{\nabla}T=0.
\end{array}
\end{equation}
\end{definition}

The above conditions are equivalent to the following system of equations which are famous as
Ambrose-Singer equations.
\begin{eqnarray}
g(T_XY,Z)+g(Y,T_XZ)=0,\label{AS1}\\
(\nabla_XR)_{YZ}=[T_X,R_{YZ}]-R_{{T_XY}Z}-R_{Y{T_XZ}},\label{AS2}\\
(\nabla_XT)_Y=[T_X,T_Y]-T_{T_XY}\label{AS3},
\end{eqnarray}
for all vector fields $X,Y,Z$.

Existence of homogeneous pseudo-Riemannian structures shows the homogeneity of the space. This fact is the subject of the following Lemma
\begin{lemma}\label{GOThm}\cite{PG}
Let $(M,g)$ be a simply connected and complete pseudo-Riemannian manifold, then
$(M,g)$ admits a pseudo-Riemannian homogeneous structure if and only if it is a reductive homogeneous pseudo-Riemannian manifold.
\end{lemma}

{\em Case1: Reductive cases:}

By applying the above lemma, we consider reductive homogeneous four-dimensional two-symmetric Lorentzian spaces. The result is the following theorem,

\begin{theorem}\label{thm1}
Every Lorentzian four-dimensional two-symmetric manifold is not reductive homogeneous.
\end{theorem}
\begin{proof}
Let $(M,g)$ be a four-dimensional two-symmetric manifold. There exist local coordinates $(x^1,\dots,x^4)$ such that the metric $g$ is defined using the Equations \eqref{g}. According to the Lemma \ref{GOThm}, $(M,g)$ is (locally) reductive homogeneous if and only if the tensor field $T$ of type $(1,2)$ exists, such that the Ambrose-Singer equations satisfy.
Let $T_{\partial_i}\partial_j= T^k_{ij }\partial_k, 1\leq i,j,k\leq 4$ be a homogeneous structure on $(M, g)$ where $\partial_1=\partial_v, \partial_2=\partial_{x_1},\partial_3=\partial_{x_2},\partial_4=\partial_u$ and
$T^k_{ij}$ are smooth functions on $M$.\\
From the Equations \eqref{AS1} and \eqref{AS2}, besides the relations between the components $T^k_{ij}$, one of the following relations for the constants $a,b,p,q,s$ must be valid:
$$
\begin{array}{l}
{\bf 1:} a=p=q=0,\\
{\bf 2:} a\neq0,q=0,s=\frac{bp}{a},\\
{\bf 3:} b=a,s=p,q=0,\\
{\bf 4:} b=-a,s=-p,
\end{array}
$$
but each of these solutions makes a contradiction with Equation \eqref{AS3} (or equivalently with $\widetilde\nabla T=0$). For example, in the case {\bf $1$}, for the components $T^k_{ij}$ have
$$
\begin{array}{l}
T_{1j}^k=0,\ 1\leq j,k\leq 4,(j,k)\neq(4,1),\ T_{14}^1=-T_{44}^4=\frac{b}{2(bx^4+s)},\\
T_{2j}^k=0,\ 1\leq j,k\leq 4,\ k\neq 1,\\
T_{3j}^k=0,\ 1\leq j,k\leq 4,\ k\neq 1,\\
T_{4j}^1=0,\ 1\leq j\leq 3,\ T_{44}^1=\frac{b(x^3)^2}{2},\\
T_{4j}^k=-T_{kj}^1,\ 1\leq j\leq 4,\ 2\leq k\leq 3,\\
T_{4j}^4=0,\ 1\leq j\leq 3.
\end{array}
$$
On the other hand, we have
$$
\begin{array}{ll}
(\widetilde\nabla_{\partial_4}T)_{44}^4=&\partial_4T_{44}^4+T_{42}^4(T_{44}^2+x^2(p+ax^4)+qx^3)+T_{43}^4(T_{44}^3+x^3(s+bx^4)+qx^2)\\
&+T_{24}^4(x^2(p+ax^4)+qx^3)+T_{34}^4(x^3(s+bx^4)+qx^2)\\
&+T_{41}^4(T_{44}^1-\frac{a(x^2)^2+b(x^3)^2}{2})-T_{14}^4\frac{a(x^2)^2+b(x^3)^2}{2}+(T_{44}^4)^2.
\end{array}
$$
If we  substitute the previous solutions in the above relation we get $(\widetilde\nabla_{\partial_4}T)_{44}^4=\frac{3b^2}{4(bx^4+s)^2}$, and so the Equation \eqref{AS3} satisfies if $b=0$ which is a contradiction, since in this case the matrix $H$ in \eqref{g0} vanishes.
\end{proof}

{\em Case2: Non-reductive cases:}

Consider a homogeneous manifold $M=G/H$ (with $H$ connected), the Lie algebra $\g$ of $G$ ,  the isotropy subalgebra $\h$, and $\m = \g /\h$ the factor space, which identifies with a subspace of $\g$ complementary to $\h$. The pair $(\g, \h)$ uniquely defines the isotropy representation
$$\psi: \g \to \mathfrak{gl}(\m), \qquad \psi(x)(y) = [x,y]_{\m} \quad \text{for all} \quad x \in \g, y \in \m.$$

Given a basis $\{h_1, ..., h_r, u_1,...,u_n\}$ of  $\g$, where $\{h_j\}$ and $ \{u_i\}$ are bases of $\h $ and $\m$, respectively, a bilinear form on $\m$ is determined by the matrix $g$ of its components with respect to the basis $\{u_i\}$ and  is invariant if and only if $^t \psi(x) \circ g + g \circ \psi(x) = 0$ for all $x \in \h$. Invariant pseudo-Riemannian metrics $g$ on the homogeneous space $M=G/H$ are in a one-to-one correspondence with nondegenerate invariant symmetric bilinear forms $g$ on $\m$ \cite{K}.

Then, $g$ uniquely defines its invariant linear Levi-Civita connection, described in terms of the corresponding homomorphism of $\h$-modules $\Lambda : \g \to \mathfrak{gl}(\m)$, such that $\Lambda (x)(y_{\m}) = [x, y]_{\m}$ for all $x\in \h, y \in \g$. Explicitly, one has
$$\begin{array}{l}
\Lambda (x)(y_{\m}) = \frac{1}{2}[x,y]_{\m} +v(x,y), \qquad \text{for all} \; x,y \in \g ,
\end{array}$$
where $v: \g \times \g \to \m $ is the $\h$-invariant symmetric mapping uniquely determined by
$$2g (v(x, y), z_{\m}) = g(x_{\m}, [z, y]_{\m}) + g(y_{\m}, [z, x]_{\m}), \qquad \text{for all} \; x, y, z \in \g.
$$
The curvature tensor is then determined by
\begin{equation}\label{curv}
\begin{array}{rcl}
R:\m \times \m & \to & \mathfrak{gl}(\m) \\[4 pt] (x, y) & \mapsto & [\Lambda(x), \Lambda(y)]-\Lambda([x, y]).
\end{array}
\end{equation}
Finally, the Ricci tensor $\varrho$ of $g$, described in terms of its components with respect to $\{u_i\}$, is given by
\begin{equation}\label{ric}
\varrho (u_i, u_j) = \sum _{r=1} ^4 R_{ri} (u_r,u_j), \qquad i,j=1,...,4
\end{equation}
and the scalar curvature $\tau$ is the trace of $\varrho$.

Non-reductive homogeneous manifolds of dimension $4$ were classified in \cite{FR}, in terms of the corresponding non-reductive Lie algebras. The corresponding invariant pseudo-Riemannian metrics, together with their connection and curvature, were explicitly described in \cite{CF,CZ}. These spaces categorized in eight classes, $A_1,\dots,A_5,B_1,B_2,B_3$. The invariant metrics of types $A_1,A_2,A_3$ can be both of Lorentzian or neutral signature while the cases $A_4,A_5$ are always Lorentzian and cases $B_1,B_2,B_3$ admit the neutral signature $(2,2)$.

\begin{theorem}\label{thm2}
Every Lorentzian four-dimensional non-reductive homogeneous manifold is locally symmetric if and only if $\nabla^2 R=0$.
\end{theorem}
\begin{proof}
Let $(M,g)$ be a Lorentzian non-reductive homogeneous four-dimensional manifold, then $(M,g)$ is isometric to a homogeneous space $G/H$ equipped with a Lorentzian invariant metric $g$, where the corresponding Lie algebras $\g$ and $\h$ are described in the cases $A_1,\dots,A_5$ of \cite{CF}.

We bring the details of the case $A_1$. The other cases can be treated in the similar way. In this case, $\g=\A _1$ is the decomposable $5$-dimensional Lie
algebra $\mathfrak{sl}(2,\mathbb R) \oplus \mathfrak{s} (2)$, where $\mathfrak s (2)$ is the $2$-dimensional solvable algebra. There exists a basis $\{e_1,...,e_5\}$ of $\A _1$, such that the non-zero products are:
$$[e_1,e_2]=2e_2, \quad [e_1,e_3]=-2e_3, \quad [e_2,e_3]=e_1, \quad [e_4,e_5]=e_4$$
and the isotropy subalgebra is $\h={\rm Span}\{h_1=e_3+e_4\}$. So, we can take
$$\m= {\rm Span}\{u_1=e_1,u_2=e_2,u_3=e_5,u_4=e_3-e_4\}.$$
With respect to $\{u_i\}$, we have the following isotropy representation $H_1$ for $h_1$ and consequently the following description of invariant metric $g$:
\begin{equation}\label{gA1}
H_1=\left(\begin{array}{cccc}
0 & -1 & 0 & 0 \\
0 & 0 & 0 & 0 \\
0 & 0 & 0 & 0 \\
1 & 0 & -\frac 12 & 0
\end{array}\right), \qquad g=\left(\begin{array}{cccc}
a & 0 & -\frac{a}{2} & 0 \\
0 & b & c & a \\
-\frac{a}{2} & c & d & 0 \\
0 & a & 0 & 0
\end{array}\right),
\end{equation}
which is nondegenerate whenever $a (a-4d)\neq 0$ and is Lorentzian if and only if $a(a-4d)<0$. Putting $\Lambda[i]:=\Lambda (u_i)$ for all indices $i=1,...,4$, we find:
{\small \begin{align}\label{lambdaA1}
\Lambda[1]=\begin{pmatrix}
0 & 0 & 0 & 0 \\ 0 & 1 & 0 & 0 \\ 0 & 0 & 0 & 0 \\ 0 & -\frac{b}{a} & -\frac{c}{a} & -1\end{pmatrix},  \qquad &\Lambda[2]=\begin{pmatrix}
0 & -\frac{8bd}{a(a - 4d)} & \frac{c}{a} &1 \\ -1 & 0 & \frac 12 & 0 \\ 0 & -\frac{4bc}{a(a - 4d)} & 0 & 0 \\ \noalign{\smallskip} -\frac{b}{a} & \frac{4bc}{a(a - 4d)} & -\frac{b}{2a} & 0\end{pmatrix}, \\
\Lambda[3]=\begin{pmatrix}
0 & \frac{c}{a} & 0 & 0 \\ \noalign{\smallskip} 0 & \frac 12 & 0 & 0 \\ 0 & 0 & 0 & 0 \\ -\frac{c}{a} & -\frac{b}{2a} &  0 & -\frac 12 \end{pmatrix}, \qquad &\Lambda[4]=0. \notag
\end{align}}
Moreover, applying \eqref{curv} and \eqref{ric}, by setting $R_{ij}=R(u_i,u_j)$, some standard calculations give that with respect to $\{u_i\}$, the non-zero curvature components are determined as following:
$$
\begin{array}{ll}
R_{12}=\left( \begin {array}{cccc} 0&{\frac {b \left( 20\,d+a \right) }{a
 \left( -4\,d+a \right) }}&-{\frac {c}{a}}&-1\\ 1&0&
-\frac{1}{2}&0\\ 0&{\frac {12b}{-4\,d+a}}&0&0
\\[2pt] {\frac {4b}{a}}&-{\frac {12cb}{a \left( -4\,d
+a \right) }}&{\frac {b}{a}}&0\end {array} \right),&
R_{13}=\left( \begin {array}{cccc} 0&-{\frac {c}{a}}&0&0
\\ 0&0&0&0\\ 0&0&0&0
\\ {\frac {c}{a}}&0&-{\frac {c}{2a}}&0
\end {array} \right),\\[20pt]
R_{14}=\left( \begin {array}{cccc} 0&-1&0&0\\ 0&0&0&0
\\ 0&0&0&0\\ 1&0&-\frac{1}{2}&0
\end {array} \right),&
R_{23}=\left( \begin {array}{cccc} 0&-{\frac {b \left( 4\,d+a \right) }
{2a \left( -4\,d+a \right) }}&-{\frac {c}{2a}}&-\frac{1}{2}
\\ \frac{1}{2}&-{\frac {c}{a}}&-\frac{1}{4}&0\\[2pt] 0&
-{\frac {2b}{-4\,d+a}}&0&0\\[2pt] -{\frac {b}{a}}&{
\frac {cb \left( -4\,d+3\,a \right) }{{a}^{2} \left( -4\,d+a \right) }
}&{\frac {{c}^{2}}{{a}^{2}}}&{\frac {c}{a}}\end {array} \right),\\[20pt]
R_{24}=\left( \begin {array}{cccc} 0&0&0&0\\ 0&-1&0&0
\\ 0&0&0&0\\ 0&{\frac {b}{a}}&{
\frac {c}{a}}&1\end {array} \right),&
R_{34}= \left( \begin {array}{cccc} 0&\frac{1}{2}&0&0\\ 0&0&0&0
\\ 0&0&0&0\\ -\frac{1}{2}&0&\frac{1}{4}&0
\end {array} \right),
\end{array}
$$
and the Ricci tensor $\varrho$ is determined by
$$\label{roA1}
\varrho =\left( \begin {array}{cccc} -2&0& 1 &0\\\noalign{\smallskip}0&\,{\frac {2b (a+12d)}{a \left(a -4d \right) }}&-\,{\frac {2c}{a}}&-2\\\noalign{\smallskip} 1 & - {\frac {2c}{a}}&-\frac 12&0\\\noalign{\smallskip}0&-2&0&0\end {array} \right).
$$
By using description of the curvature tensor and Levi-Civita connection, we set $\nabla R=0$ and have, the homogeneous spaces $G/H$ is locally symmetric if and only if $b=0$. Also, $\nabla^2R=0$ if and only if $b=0$, so we conclude that the homogeneous spaces of type $A_1$ are locally symmetric if and only if $\nabla^2R=0$. Similar arguments will be applied for the other Lorentzian cases and this finishes the proof.
\end{proof}

The following remark is the direct conclusion of the Theorems \ref{thm1} and \ref{thm2}.
\begin{rem}
Let $(M,g)$ be a homogeneous Lorentzian four-dimensional manifold, then $(M,g)$ is not a two-symmetric space.
\end{rem}

Classification of four-dimensional pseudo-Riemannian homogeneous spaces with non-trivial isotropy has been studied in \cite{K}  in order to find the local classification of four-dimensional Einstein-Maxwell homogeneous spaces with an invariant pseudo-Riemannian metric of arbitrary signature. The presented list is a good reference to study pseudo-Riemannian homogeneous four-manifolds. We apply the mentioned classification to find an example of a pseudo-Riemannian two-symmetric homogeneous four-manifold, equipped with an invariant metric of {\em neutral} signature.
\begin{ex}
Let $M=G/H$ be a homogeneous four-dimensional manifold of type $1.4^1:9$ of the Komrakov's list \cite{K}. In this case, the Lie algebras $\g$ and $\h$ are described as following:
$$
\begin{array}{ll}
\g=\left\{\left.\left(\begin{array}{cccc}x&0&0&0\\0&0&\lambda x&0\\0&x&x&0\\0&0&x&-\mu x\end{array}\right)\right|x\in\R,\lambda,\mu\in\R\right\}\rightthreetimes(\mathfrak n_3\times \R),& \h=\langle p\rangle,
\end{array}
$$
where $\mathfrak n_3=\langle h,p,q\rangle$ with the only non-zero bracket $[p,q]=h$. If $\g={\rm span}\{u_1,u_2,u_3,u_4\}$ and $\h={\rm span}\{h_1\}$, the table of Lie brackets is:
$$
\begin{array}{c|ccccc}
[\ ,\ ]&h_1&u_1&u_2&u_3&u_4\\
\hline
h_1&0&0&u_1&u_2&0\\
u_1&0&0&0&u_1&0\\
u_2&-u_1&0&0&\lambda h_1+u_2+u_4&0\\
u_3&-u_2&-u_1&-\lambda h_1-u_2-u_4&0&\mu u_4\\
u_4&0&0&0&-\mu u_4&0
\end{array}
$$
The invariant metric will be calculated as following
$$
g= \left( \begin {array}{cccc} 0&0&-a&0\\ 0&a&0&0
\\ -a&0&b&c\\0&0&c&d
\end {array} \right),
$$
for arbitrary real constants $a\neq 0,b,c,d$. This metric admits both Lorentzian and neutral signatures while for $d=-9a$, $g$ is of neutral signature. We also set $\mu=-\frac 52$ and $\lambda=\frac 34$. Keeping in mind $\Lambda[i]=\Lambda(u_i)$ for all indices $i=1,\dots,4$, the components of the Levi-Civita connection are:
$$
\begin{array}{ll}
\Lambda[1]=0,&\Lambda[2]=\left( \begin {array}{cccc} 0&1&{\frac {c}{2a}}&-\frac 92
\\ 0&0&1&0\\ 0&0&0&0
\\ 0&0&\frac 12&0\end {array} \right),\\
\Lambda[3]= \left( \begin {array}{cccc} -1&{\frac {c}{2a}}&{\frac {6\,ba
-{c}^{2}}{6{a}^{2}}}&{\frac {5c}{2a}}\\[4pt] 0&0&{
\frac {c}{a}}&-\frac 92\\ 0&0&1&0\\ 0&-
\frac 12&-{\frac {c}{6a}}&0\end {array} \right),&
\Lambda[4]=\left( \begin {array}{cccc} 0&-\frac 92&{\frac {5c}{2a}}&-{\frac {45}{2
}}\\[4pt] 0&0&-\frac 92&0\\ 0&0&0&0
\\ 0&0&\frac 52&0\end {array} \right),
\end{array}
$$
also, by using Equation \eqref{curv}, if set $R_{ij}=R(u_i,u_j)$, the non-zero components of the curvature tensor are
$$
\begin{array}{ll}
R_{23}=\left( \begin {array}{cccc} 0&6&-{\frac {2c}{a}}&18
\\ 0&0&6&0\\ 0&0&0&0
\\ 0&0&-2&0\end {array} \right),&
R_{34}=\left( \begin {array}{cccc} 0&-18&{\frac {6c}{a}}&-54
\\ 0&0&-18&0\\ 0&0&0&0
\\ 0&0&6&0\end {array} \right).
\end{array}
$$
The space is Ricci flat and the only non-zero covariant derivative of the curvature tensor is in the direction of $u_3$. We set $(\Lambda[k]R)_{ij}=(\Lambda(u_k)R)(u_i,u_j)$, the non-zero components are
$$
\begin{array}{ll}
(\Lambda[3]R)_{23}=\left( \begin {array}{cccc} 0&6&-{\frac {2c}{a}}&18
\\ 0&0&6&0\\ 0&0&0&0
\\ 0&0&-2&0\end {array} \right),&
(\Lambda[3]R)_{34}=\left( \begin {array}{cccc} 0&-18&{\frac {6c}{a}}&-54
\\ 0&0&-18&0\\ 0&0&0&0
\\ 0&0&6&0\end {array} \right).
\end{array}
$$
Clearly, $(M=G/H,g)$ is never locally symmetric but by straight forward calculations we get $\nabla^2R=0$ which shows that the spaces is two-symmetric.
\end{ex}

\end{document}